\documentclass[12pt]{article}
\usepackage{graphicx}
\usepackage{epstopdf}
\usepackage{subfigure}
\usepackage{color}
\usepackage[svgnames]{xcolor}
\usepackage[english]{babel}
\usepackage{amssymb,amsmath,amsthm}
\usepackage{amsmath,amssymb}
\usepackage{array}
\usepackage{multirow}
\usepackage{amssymb}
\usepackage{isomath}
\usepackage{blkarray}
\usepackage{cancel}
\usepackage{extarrows}
\usepackage{lscape}

\usepackage{amsfonts}
\usepackage{booktabs}
\usepackage{algorithm}
\usepackage{algorithmic}
\usepackage{pifont}
\usepackage{bm} 
\usepackage{authblk}
\usepackage{appendix}
\usepackage{mathtools}

\usepackage{caption} 
\usepackage{graphicx}
\usepackage{subfigure}
\usepackage{tabularx}
\usepackage[colorlinks=true, linkcolor=red, citecolor=green]{hyperref}
\usepackage{url}

\usepackage{amsmath}

\newtheorem{thm}{Theorem}[section]
\newtheorem{defi}{Definition}[section]

\newtheorem{lem}[thm]{Lemma}

\newtheorem{coro}{Corollary}[section]

\newtheorem{Remark}{Remark}[section]

\newtheorem{Problem}{Problem}[section]

\title{Partial-Twuality Polynomials of Paired Matrices}

\author[a]{Xiaoxiang Yu}
\author[a]{Rong-Xia Hao$^{*}$}
\author[a]{Jianbing Liu\footnote{Corresponding authors: rxhao@bjtu.edu.cn,  jbliu1@bjtu.edu.cn}}
\affil[a]{\footnotesize School of Mathematics and Statistics, Beijing Jiaotong University, Beijing 100044, P. R. China}
\date{}

\begin{document}
\baselineskip 0.65cm
\maketitle

\vspace{-1.2em}
\begin{abstract}
Gross, Mansour, and Tucker~[European Journal of Combinatorics, 95 (2021): 103329] introduced the \emph{partial-twuality polynomials} of ribbon graphs. Recently, Deng, Jin, and Yan generalized the partial-twuality polynomials to the framework of matrix algebra and investigated several of their basic properties. They asked whether there exist matrix operations, called partial duality $\delta$ and partial Petrie duality $\tau$, on pairs $(M,A)$, where $M$ is a square matrix whose rows and columns are indexed by a finite set $V$ and $A\subseteq V$, such that 
$\delta^2=\tau^2=(\delta\tau)^3=id$ and the exponent of the partial-twuality polynomials coincides with some parameter of the matrix obtained by applying \(\bullet\) to \((M, A)\). In this paper, we introduce a paired-matrix framework for partial-twuality polynomials over the binary field $\mathbb{GF}(2)$. We prove that there exist two local operations \(\delta\) and \(\tau\) on \(\bigl((M,I_{|V|}),A\bigr)\) satisfying $\delta^2=\tau^2=(\delta\tau)^3=id$ and $P_{\langle \bullet \rangle}((M,I_{|V|}),z)=P_{\langle \bullet \rangle}(M,z)$ for $\bullet \in \{\delta, \tau, \delta\tau, \tau\delta, \delta\tau\delta\}$, thereby answering their question affirmatively. Finally, we establish a recurrence relation for the partial \(\langle\delta\tau\delta\rangle\)-polynomial with respect to an edge. This recurrence enables the computation of the partial \(\langle\delta\tau\delta\rangle\)-polynomial for certain bouquets, simple graphs, and simple signed graphs.

\vspace{3mm}


\end{abstract}

\section{Introduction}
A ribbon graph~\cite{2002_Ribbon graph} is a surface with boundary representing a graph cellularly embedded in a surface.~The Petrie operation was introduced by Wilson in 1979~\cite{Introduction_Petrial}. Geometric duality \(\delta\) and Petrial duality \(\tau\) are involutions on ribbon graphs, while their composition \(\delta\tau\) has order three. Consequently, these operations generate an action of a group isomorphic to the symmetric group \(S_3\). Its five nonidentity elements, $\delta, \tau, \delta\tau, \tau\delta, \text{and } \delta\tau\delta,$
are collectively called \emph{twualities}~\cite{Twuality}. Chmutov~\cite{Partial} introduced partial duality for ribbon graphs, and Ellis-Monaghan and Moffatt~\cite{Partial_Twuality} extended this notion to partial twualities.

In 2021, Gross, Mansour, and Tucker~\cite{2021_II} introduced the \emph{partial-$\langle \bullet \rangle$ polynomial}, which enumerates the partial-\(\bullet\) twualities of a ribbon graph according to their Euler genera. More precisely, for \(\bullet\in\{\delta,\tau,\delta\tau,\tau\delta,\delta\tau\delta\}\), the polynomial records the values \(\varepsilon(G^{\bullet\mid A})\), where \(A\) ranges over the subsets of the edge set of \(G\).

\begin{defi}\emph{(\cite{2021_II})}\label{poly_Petrial}
For $\bullet \in \{\delta, \tau, \delta\tau, \tau\delta, \delta\tau\delta\}$, the \emph{partial-$\langle \bullet \rangle$ polynomial} of
any ribbon graph $G$ is the generating function 
$$^{\partial}{\varepsilon^{\bullet}_{G}}(z)=\sum\limits_{A\subseteq E(G)}z^{\varepsilon{(G^{\bullet|A})}}$$
\noindent that enumerates partial-$\bullet$ of $G$ by Euler genus. 
\end{defi}

A \textit{bouquet} is a ribbon graph with a single vertex. In  \cite{Matroids_Yan1} and \cite{Matroids_Yan}, Yan and Jin introduced the notations of partial-twuality polynomials for delta-matroids. 
For normal binary delta-matroids, which serve as a generalization of bouquets, these polynomials are completely determined by signed intersection graphs. 
Cheng \cite{fourterm_chengzhiyun} built a relationship between the Euler genus of an orientable bouquet and the rank of the adjacency matrix of its intersection graph, and generalized the partial-$\langle \delta \rangle$ polynomial from intersection graphs to all simply graphs without referring to embeddings.

Loop complementation on delta-matroids~\cite{pivotandloocomplementation} is compatible with the partial Petrial operation on ribbon graphs~\cite{Matroids_Yan}. Matrix loop complementation and its relation to delta-matroids were studied systematically in~\cite{pivotandloocomplementation}. There is also a close connection between the partial duality operation on ribbon graphs, twist on delta-matroids, and the principal pivot transform of matrices for any nonsingular principal minor \cite{twistpivot_Moffatt}. 
Yan and Jin studied partial-twuality polynomials for delta-matroids as analogues of those for ribbon graphs \cite{Matroids_Yan}. 

Recently, Deng, Jin and Yan \cite{Matrice_Twualitypolynomial} proved that the Euler genus of a bouquet under a partial-twuality operation is a linear combination of ranks and/or coranks of the adjacency matrices of a related signed intersection graph, and they gave the \emph{partial-$\langle \bullet \rangle$ polynomial} of a square matrix as a generalization of the partial-$\langle \bullet \rangle$ polynomial of a bouquet as follows.
\begin{defi} \emph{\cite{Matrice_Twualitypolynomial} }\label{def:matrix_partial_twuality}
	Let $M$ be a square matrix over a field, with rows and columns indexed by a finite set $V$. For $\bullet \in \{\delta, \tau, \delta\tau, \tau\delta, \delta\tau\delta\}$, the \emph{partial-$\langle \bullet \rangle$ polynomial} of $M$ is defined as
	\[P_{\langle \bullet \rangle}(M, z) = \sum_{A \subseteq V} z^{r_{\langle \bullet \rangle}(M, A)},\]
	where the exponent function $r_{\langle \bullet \rangle}(M, A)$ is given by
	\vspace{-0.6em}
	\begin{align*}
		r_{\langle \delta \rangle}(M, A) &= \operatorname{rank}(M[A]) + \operatorname{rank}(M[A^c]), \\
		r_{\langle \tau \rangle}(M, A) &= \operatorname{rank}(M + I_A), \\
		r_{\langle \delta \tau \rangle}(M, A) &= \operatorname{rank}((M + I_A)[A]) + \operatorname{rank}(M[A^c]), \\
		r_{\langle \tau \delta \rangle}(M, A) &= \operatorname{rank}(M + I_A) - \operatorname{corank}(M[A]), \\
		r_{\langle \delta\tau\delta\rangle}(M, A) &= \operatorname{rank}(M) - \operatorname{corank}((M + I_A)[A]). 
	\end{align*}
	
	\vspace{-0.3em}
\noindent Here $M[A]$ denotes the principal submatrix of $M$ on $A \subseteq V$, and $A^c = V \setminus A$. 
\end{defi}
Deng, Jin and Yan~\cite{Matrice_Twualitypolynomial} also established several fundamental properties of these polynomials, including product formulas, leaf-reduction recursions, and interpolation properties. They further proved that the partial-\(\langle\delta\rangle\) polynomial is invariant under matrix pivoting. Motivated by the research void regarding singular matrices in $\delta$ operations, for a square matrix $M$ with rows and columns indexed by a finite set $V$ and $A\subseteq V$, they posed the following problem.
\begin{Problem}\emph{\cite{Matrice_Twualitypolynomial}}\label{Problem1}
	Do there exist two matrix operations \(\delta\) and \(\tau\) on pairs \((M, A)\), where \(M\) is a square matrix and \(A\) is a subset of its index set, that satisfy the relations \(\delta^2 = \tau^2 = (\delta \tau)^3 = \text{id}\), generating an \(S_3\) action on such pairs, and such that for each \(\bullet \in \{\delta, \tau, \delta \tau, \tau \delta, \tau \delta \tau\}\), the exponent \(r_{\langle \bullet \rangle}(M, A)\) coincides with some parameter of the matrix obtained by applying \(\bullet\) to \((M, A)\)\emph{?}
\end{Problem}



In this paper, for any $n\times n$ matrix $M$, 
we prove that there exist two local operations \(\delta\) and \(\tau\) on \(\bigl((M,I_{|V|}),A\bigr)\) satisfying $\delta^2=\tau^2=(\delta\tau)^3=id$ (Theorem \ref{Condition1}) and $P_{\langle \bullet \rangle}((M,I_{|V|}),z)=P_{\langle \bullet \rangle}(M,z)$ for $\bullet \in \{\delta, \tau, \delta\tau, \tau\delta, \delta\tau\delta\}$ (Theorem  \ref{Th:polyoftwuality}), which gives an answer of Problem \ref{Problem1}. Finally, we establish three recurrence relations for the partial-$\langle \delta\tau\delta \rangle$ polynomial of simple signed graphs associated with an edge (Theorems \ref{recurrence1}-\ref{recurrence3}). 

The rest of this paper is organized as follows. In section $2$, we introduce the necessary notations and foundational properties used throughout the paper. In section $3$, proofs of Theorems \ref{Condition1} and \ref{Th:polyoftwuality} are given. In section $4$, we devote to the proofs of Theorems \ref{recurrence1}-\ref{recurrence3}.

\section{Preliminaries}

Throughout this paper, all graphs are finite and undirected, and all matrices are over the binary field $\mathbb{GF}(2)$, unless stated otherwise. This section introduces the definitions, notation, and lemmas used in the proofs of the main results.

\subsection{Ribbon graph and simple signed graph}
A {\em ribbon graph} $G$ is a surface with boundary represented as the union of two sets of closed topological discs, called vertex-discs $V(G)$ and edge-ribbons $E(G)$, satisfying the following three conditions \cite{2002_Ribbon graph}: (1) vertex-discs and edge-ribbons intersect in disjoint line segments; 
(2) each such line segment lies on the boundary of precisely one vertex-disc and precisely one edge-ribbon; (3) every edge-ribbon contains exactly two such line segments.


For a ribbon graph \(G\) and \(A\subseteq E(G)\), the \emph{partial dual} \(G^{\delta\mid A}\) is obtained by gluing a disc to \(G\) along each boundary component of the spanning ribbon subgraph \((V(G),A)\), deleting the interiors of the original vertex-discs, and leaving the edge-ribbons unchanged~\cite{Partial}. The geometric dual of \(G\) is \(G^{\delta\mid E(G)}\). The \emph{Petrie} of \(G\) is obtained by detaching one end of every edge-ribbon from its incident vertex-disc, introducing a half-twist, and then reattaching it \cite{Introduction_Petrial}. If this operation is performed only on the edges in \(A\subseteq E(G)\), the resulting ribbon graph is the \emph{partial Petrial} \(G^{\tau\mid A}\).\\ 
\indent A loop of a bouquet is \emph{twisted} (resp.~\emph{untwisted}) if its edge-ribbon is homeomorphic to a M\"obius band (resp.~an annulus). The \emph{signed rotation} of a bouquet is the cyclic order of its half-edges at the vertex-disc, together with signs assigned to the half-edges: the two half-edges of an untwisted loop have the same sign, whereas those of a twisted loop have different signs. Two loops of a bouquet are \emph{interlaced} if their half-edges alternate in the signed rotation.\\
\indent A \emph{signed graph} is a graph whose vertices are assigned signs \(+\) or \(-\). A \emph{simple signed graph} is a signed graph whose underlying graph is simple. The \emph{intersection graph} \(I(B)\) of a bouquet \(B\) has the loops of \(B\) as its vertex set, with two vertices adjacent precisely when the corresponding loops are interlaced. The \emph{signed intersection graph} \(SI(B)\) is obtained from \(I(B)\) by assigning sign \(+\) to an untwisted loop and sign \(-\) to a twisted loop. Its adjacency matrix, denoted by \(\operatorname{adj}(SI(B))\), is obtained from \(\operatorname{adj}(I(B))\) by placing a \(1\) in the diagonal position corresponding to every negative vertex. Yan and Jin~\cite{Matroids_Yan} proved that the partial-\(\langle\bullet\rangle\) polynomials of bouquets are completely determined by their signed intersection graphs.

Following Moffatt \cite{graft}, a \emph{graft} \((G,L)\) consists of a simple graph \(G\) together with a subset \(L\subseteq V(G)\). Its adjacency matrix~$\text{adj}{(G,L)}$ is the matrix with rows and columns indexed by $V(G)$ defined by
	\[
	\text{adj}{(G,L)}{(u,v)} = 
	\begin{cases}
		1, & \text{if } uv \in E(G) \text{ and } u \neq v, \\
		1, & \text{if } u = v \in L, \\
		0, & \text{otherwise},
	\end{cases}
	\]
over $\mathbb{GF}(2)$. Deng, Jin, Yan \cite{Matrice_Twualitypolynomial} generalized the partial-$\langle \bullet \rangle$ polynomial of bouquets to grafts by  restricting the arbitrary matrix $M$ in Definition \ref{def:matrix_partial_twuality} to the adjacency matrix of a graft. 

For a graph \(G\) and a vertex \(v\in V(G)\), let \(N_G(v)\) denote the set of vertices adjacent to \(v\). The \emph{local complementation} of a simple graph \(G\) at \(v\), denoted by \(G\ast v\), is obtained by toggling all adjacencies between distinct vertices of \(N_G(v)\)~\cite{Kotzig1968}. Let $G$ be a simple signed graph with vertex subset $L \subseteq V(G)$. For a graft \((G,L)\) and any \(v\in V(G)\), define local complementation at \(v\) by
\[(G,L)\ast v:=(G\ast v,\,L\triangle N_G(v)).\]
In the original definition in~\cite{graft}, local complementation is applied at a negative vertex because of its connection with matrix pivoting. Here we allow the operation at both positive and negative vertices and use it solely as a graph operation.

\begin{Remark}\label{keylemma}
Let \((G,L)\) be a graft and let \(v\in V(G)\). Order the vertices with \(v\) first, and write
\[
\operatorname{adj}(G,L)=
\begin{bmatrix}
a & B^{T}\\
B & C
\end{bmatrix},
\]
where \(B\) is the neighborhood column vector of \(v\). Then
\[\operatorname{adj}((G,L)\ast v)=
\begin{bmatrix}
a & B^{T}\\
B & C+BB^{T}
\end{bmatrix}.\]
\end{Remark}

\subsection{Matrices and principal pivot transforms}

For \(n\in\mathbb N\), let \(I_n\) denote the \(n\times n\) identity matrix. If \(V\) is a finite index set and \(A\subseteq V\), let \(I_A\) denote the diagonal \(|V|\times |V|\) matrix with diagonal entry \(1\) on \(A\) and \(0\) otherwise.

We use standard delta-matroid terminology. Let \(E\) be a finite set, and let \(C\) be a symmetric \(|E|\times|E|\) matrix over \(\mathbb{GF}(2)\), with rows and columns indexed by \(E\). Define
\[D(C)=(E,\mathcal F(C)),\qquad
\mathcal F(C):=\{A\subseteq E:C[A]\text{ is nonsingular}\}.\] 
We adopt the convention that the empty matrix is non-singular. For $X \subseteq E$, the \emph{twist} of $D(C)$ with respect to $X$, denoted by $D(C)\ast  X$, is given by $D(C)\ast X \coloneqq (E, \{A \triangle X : A \in \mathcal{F}(C)\}).$ 
A delta-matroid \(D\) on \(E\) is \emph{binary} if \(D\ast X=D(C)\) for some \(X\subseteq E\) and some symmetric matrix \(C\) over \(\mathbb{GF}(2)\). In particular, a normal binary delta-matroid has a direct representation of the form \(D(C)\). Every ribbon-graphic delta-matroid is binary~\cite{RepresentabilityofDeltamatroids}.~For a bouquet \(B\), the delta-matroid associated with \(B\) is represented by the adjacency matrix of its signed intersection graph~\cite{Matrice_Twualitypolynomial}.

Let \(X\subseteq E\), and suppose that \(C\) is written in block form, with rows and columns ordered by \(X\) and \(E\setminus X\), as $C=
\begin{bmatrix}
\alpha & \beta\\
\gamma & \mu
\end{bmatrix}.$ 
If \(C[X]=\alpha\) is nonsingular, the \emph{principal pivot transform} of \(C\) with respect to \(X\), denoted by \(C\ast X\), is
\[C\ast X=
\begin{bmatrix}
\alpha^{-1} & \alpha^{-1}\beta\\
-\gamma\alpha^{-1} & \mu-\gamma\alpha^{-1}\beta
\end{bmatrix}.\]
Over \(\mathbb{GF}(2)\), the displayed minus signs agree with plus signs. Bouchet~\cite{RepresentabilityofDeltamatroids} proved that
\[D(C\ast X)=D(C)\ast X.\]
Thus, the principal pivot transform realizes the twist operation whenever the relevant principal submatrix is nonsingular. If \(C[X]\) is singular, then \(C\ast X\) is not defined.

We conclude with the standard Schur-complement rank formula.

\begin{lem}\emph{\cite{SchurComplement}}\label{calculaterank}
	Suppose that 
	$M = \begin{bmatrix} 
		A & B \\ C & D 
	\end{bmatrix}$.
	If $A$ is a non-singular matrix, then 
	\[\operatorname{rank}(M)=\operatorname{rank}(A)+\operatorname{rank}(D-CA^{-1} B).\]
\end{lem}

\section{The $\delta$ and $\tau$ operations on paired matrices}

In this section, we generalize partial duality \(\delta\) and partial Petrial duality \(\tau\) from bouquets to paired matrices. For every \(n\times n\) matrix \(M\), we define local operations \(\delta\) and \(\tau\) on an \(n\times2n\) paired matrix associated with \(M\). These operations provide an affirmative answer to Problem~\ref{Problem1}.

Let \(M\) and \(N\) be \(n\times n\) matrices. The paired matrix \((M,N)\) denotes the block matrix \([M\mid N]\), with left block \(L(M,N):=M\) and right block \(R(M,N):=N\). Suppose that the rows and columns of \(M\) and \(N\) are indexed by a finite set \(V\) in a fixed order. For \(x\in V\), let \(M_x\) and \(N_x\) be the columns indexed by \(x\) in \(M\) and \(N\), respectively. We call \((L,R)_x:=[M_x\mid N_x]\) the \emph{column pair of $(L,R)$ at \(x\)}.

\begin{defi}\label{defdeltatau}
Let \((L,R)\) be a paired matrix indexed by \(V\). For \(x\in V\), define local operations \(\delta_x\) and \(\tau_x\) by
\[
\bigl(\delta_x(L,R)\bigr)_y=
\begin{cases}
(R_y,L_y), & y=x,\\
(L_y,R_y), & y\ne x,
\end{cases}
\qquad
\bigl(\tau_x(L,R)\bigr)_y=
\begin{cases}
(L_y+R_y,R_y), & y=x,\\
(L_y,R_y), & y\ne x.
\end{cases}
\]
For a word \(w=w_1w_2\cdots w_k\) over \(\{\delta,\tau\}\), define
\[
w_x:=w_{1,x}\circ w_{2,x}\circ\cdots\circ w_{k,x},
\]
so that the rightmost operation acts first. In particular,
\((\delta\tau)_x=\delta_x\circ\tau_x\), \((\tau\delta)_x=\tau_x\circ\delta_x\), and \((\delta\tau\delta)_x=\delta_x\circ\tau_x\circ\delta_x\). For \(A\subseteq V\), set
\[
\delta_A:=\prod_{x\in A}\delta_x,\qquad
\tau_A:=\prod_{x\in A}\tau_x,\qquad
\bullet_A:=\prod_{x\in A}\bullet_x,
\]
where \(\bullet\in\{\delta,\tau,\delta\tau,\tau\delta,\delta\tau\delta\}\).
\end{defi}

\noindent\textbf{Note.} The products in Definition~\ref{defdeltatau} are well-defined because local operations at distinct indices act on disjoint column pairs and therefore commute.

\begin{thm}\label{Condition1}
	For every $A\subseteq V$, 
	\[
	\delta_A^2=\tau_A^2=(\delta_A\tau_A)^3=id.
	\]
\end{thm}

\begin{proof}
	Assume that the paired matrix is $(L,R)$. Since the operations performed on distinct elements of $A$ are mutually independent, it suffices to prove that	$\delta_x^2=\tau_x^2=(\delta_x\tau_x)^3=id$ for any $x\in A$.  By Definition \ref{defdeltatau}, one has 
	\[(\delta_x(L,R))_x=(R_x,L_x) ,\qquad(\tau_x(L,R))_x=(L_x+R_x,R_x)\]
	and other column pairs indexed by $y\in A\backslash\{x\}$ remain unchanged. It is sufficient to consider the action of $\delta, \tau $ and $\delta\tau$ on the column pair indexed by $x$. For convenience, we represent the entire matrix $\bullet_x(L,R)$ by the column pair indexed by $x$, where $\bullet$ is a operator word over the alphabet $\{\delta,\tau\}$. Thus we have that
\[\delta_x^2(L,R)=\delta_x(R_x,L_x)=(L_x,R_x),\] 
\[\tau_x^2(L,R)=\tau_x(L_x+R_x,R_x)=(L_x+R_x+R_x,R_x)=(L_x,R_x)\] 
over $\mathbb{GF}(2)$, which implies that $\delta_x^2=\tau_x^2=id$.
	
Computing the product $\delta\tau$, we can obtain that
	\[\begin{aligned}
		(\delta\tau)_x(L,R)=\delta_x(L_x+R_x,R_x)=(R_x,L_x+R_x),
	\end{aligned}\]
	\[\begin{aligned}
		(\delta\tau)_x^2(L,R)=(\delta\tau)_x(R_x,L_x+R_x)=(L_x+R_x,L_x).
	\end{aligned}\]
	Applying it a third time gives
	$\begin{aligned}
		(\delta\tau)_x^3(L,R)=(\delta\tau)_x(L_x+R_x,L_x)=(L_x,R_x),
	\end{aligned}$ 
	it follows that $(\delta\tau)_x^3=id$.
	
	Therefore $\delta_A^2=\tau_A^2=(\delta_A\tau_A)^3=id.$
\end{proof}

\begin{defi}\label{def:twualitypoly(M,I)}
	Let~$M$ be an~$n\times n$ matrix over $\mathbb{GF}(2)$,~with rows and columns indexed by a set $V$.~For $\bullet \in \{\delta, \tau, \delta\tau, \tau\delta, \delta\tau\delta\}$,~the \emph{partial-$\langle \bullet \rangle$ polynomial} of the paired matrix $(M,I_n)$ is defined as
	\[P_{\langle \bullet \rangle}((M,I_n),z) = \sum_{A \subseteq V} z^{\text{rank}(L(\bullet_A(M,I_n)))+\text{rank}(R(\bullet_A(M,I_n)))-n}.\]
	
\end{defi}

\begin{thm}\label{Th:polyoftwuality}
	Let $M$ be an $n\times n$ matrix over $\mathbb{GF}(2)$ with rows and columns indexed by a finite set $V$. For $\bullet \in \{\delta, \tau, \delta\tau, \tau\delta, \delta\tau\delta\}$, we have
	\[P_{\langle \bullet \rangle}((M,I_n),z)=P_{\langle \bullet \rangle}(M,z).\]
\end{thm}
\begin{proof}
	For any $A\subseteq V$, denote 
	\begin{equation}\label{eq0}
		\text{rank}(L(\bullet_A(M,I_n)))+\text{rank}(R(\bullet_A(M,I_n)))-n=r_{\langle \bullet \rangle}((M,I_n),A).
	\end{equation}
	By Definition \ref{def:twualitypoly(M,I)}, it suffices to prove that 
	$r_{\langle \bullet \rangle}((M,I_n),A)=r_{\langle \bullet \rangle}(M,A)$. Assume that $|A|=a$.
	
	\textbf{Case 1. $\bullet=\delta$} 
	
	For $\delta$, by Definition \ref{defdeltatau}, one has
	\[	\begin{aligned}
	\delta_A(M,I_n) =\delta_A \left(\!\!
	\begin{array}{c@{}c}		
		\begin{array}{c}
		\end{array} &
		\left[\!\!
		\begin{array}{cc|cc}
			M[A] \!\!&\!\! M[A,A^c] \!\!&\!\! I_{a} \!\!&\!\! \mathbf{0}_{a,n-a}\\
			M[A^c,A] \!\!&\!\! M[A^c] \!\!&\!\! \mathbf{0}_{n-a,a} \!\!&\!\! I_{n-a}\\
		\end{array}
		\!\!\right]
	\end{array}\!\!\right)=\left[\!\!
		\begin{array}{cc|cc}
			I_{a} \!\!&\!\! M[A,A^c] \!\!&\!\! M[A] \!\!&\!\! \mathbf{0}_{a,n-a}\\
			\mathbf{0}_{n-a,a} \!\!&\!\! M[A^c] \!\!&\!\! M[A^c,A] \!\!&\!\! I_{n-a}\\
		\end{array}
		\!\!\right].
	\end{aligned}\]
	By Equation (\ref{eq0}) and Definition \ref{def:matrix_partial_twuality}, we have that \[r_{\langle \delta \rangle}\!((M \!,I_n),A) =  \text{rank}(M[A])  + \text{rank}(M[A^c]) =
	r_{\langle \delta \rangle}\!(M \!,A).\]
	
	\textbf{Case 2. $\bullet=\tau$ } 
	
	For $\tau$, by Definition \ref{defdeltatau}, we have that
	\[	\begin{aligned}
		\!\tau_A(M,I_n) \!=\! \tau_A\! \left(\!\!\!\!
		\begin{array}{c@{}c}
			\begin{array}{c}
			\end{array} &
			\left[\!\!\!
			\begin{array}{cc|cc}
				\!M[A] \!\!&\!\!\! M[A,A^c] \!\!\!&\!\! I_{a} \!\!\!&\!\!\! \mathbf{0}_{a,n-a}\\
				\!M[A^c\!\!,A] \!\!\!&\!\!\!\! M[A^c] \!\!\!&\!\! \mathbf{0}_{n-a,a} \!\!\!&\!\!\! I_{n-a}\\
			\end{array}
			\!\!\!\right]
		\end{array}\!\! \!\!\right)
		\!\!=\!\!\! \!\begin{array}{c@{}c}
			\left[\!\!\!
			\begin{array}{cc|cc}
				\!M[A]\!+\! I_{a} \!\!&\!\!\! M[A,A^c] \!\!&\!\! I_{a} \!\!\!&\!\!\! \mathbf{0}_{a,n-a}\\
				\! M[A^c\!,A] \!\!&\!\!\! M[A^c] \!\!&\!\! \mathbf{0}_{n-a,a} \!\!\!&\!\!\! I_{n-a}\\
			\end{array}
			\!\!\!\right]
		\end{array}
		\!\!=\!\!(M\!+\!I_A,I_n),\\
	\end{aligned}
	\]
	where $I_A=\left[\!\!\!
	\begin{array}{cc}
		  I_{a} &  \mathbf{0}_{a,n-a}\\
		 \mathbf{0}_{n-a,a} & \mathbf{0}_{n-a}\\
	\end{array}
	\!\!\!\right]$.
	By Equation (\ref{eq0}) and Definition (\ref{def:matrix_partial_twuality}), one has
	$$r_{\langle \tau \rangle}((M,I_n),A)=\text{rank}(M+I_A)=
	r_{\langle \tau \rangle}(M,A).$$
	
	\textbf{Case 3. $\bullet=\delta\tau$ } 
	
	For $\delta\tau$, one has that
	\[	\begin{aligned}
		(\delta\tau)_A(M,I_n) &=\delta_A \left(\!\!\!
		\begin{array}{c@{}c}
			\left[\!
			\begin{array}{cc|cc}
				M[A]\!+\! I_a \!\!&\!\! M[A,A^c] \!\!&\!\! I_{a} \!\!\!&\!\!\! \mathbf{0}_{a,n-a}\\
				M[A^c,A] \!\!&\!\! M[A^c] \!\!&\!\! \mathbf{0}_{n-a,a} \!\!\!&\!\!\! I_{n-a}\\
			\end{array}
			\!\right]
		\end{array}\!\right)\!\!=\!\!\begin{array}{c@{}c}
			\left[\!
			\begin{array}{cc|cc}
				I_{a} \!\!\!&\!\!\! M[A,A^c] \!\!&\!\! M[A]\!+\! I_a \!\!&\!\! \mathbf{0}_{a,n-a}\\
				\mathbf{0}_{n-a,a} \!\!\!&\!\!\! M[A^c] \!\!&\!\! M[A^c,A] \!\!&\!\! I_{n-a}\\
			\end{array}
			\!\right]
		\end{array}\\
	\end{aligned}
	\]
	by Definition \ref{defdeltatau}. Combined with Equation (\ref{eq0}) and Definition \ref{def:matrix_partial_twuality}, one has that 
	\[r_{\langle \delta\tau \rangle}((M,I_n),A)=\text{rank}((M+I_n)[A])+\text{rank}(M[A^c])=
	r_{\langle \delta\tau \rangle}(M,A).\]
	
	\textbf{Case 4. $\bullet=\tau\delta$ } 
	
	For $\tau\delta$, one has that
	\[	\begin{aligned}
		(\tau\delta)_A(M,I_n) \!=\!\tau_A \left(\!\!\!
		\begin{array}{c@{}c}
			\left[\!\!
			\begin{array}{cc|cc}
				I_{a} \!\!\!&\!\!\! M[A,A^c] \!\!&\!\! M[A] \!\!\!&\!\!\! \mathbf{0}_{a,n-a}\\
				\mathbf{0}_{n-a,a} \!\!\!&\!\!\! M[A^c] \!\!&\!\! M[A^c,A] \!\!\!&\!\!\! I_{n-a}\\
			\end{array}
			\!\!\right]
		\end{array}\!\right)\!\!=\!\!\! \begin{array}{c@{}c}
			\left[\!\!
			\begin{array}{cc|cc}
				M[A]+I_a \!\!&\!\! M[A,A^c] \!\!&\!\! M[A] \!\!&\!\! \mathbf{0}_{a,n-a}\\
				M[A^c,A] \!\!&\!\! M[A^c] \!\!&\!\! M[A^c,A] \!\!&\!\! I_{n-a}\\
			\end{array}
			\!\!\right]
		\end{array}\\
	\end{aligned}
	\]
	by Definition \ref{defdeltatau}. Combined with Equation (\ref{eq0}) and Definition \ref{def:matrix_partial_twuality}, we have that 
	$$r_{\langle \tau\delta \rangle}((M,I_n),A)=\text{rank}(M+I_{A})+\text{rank}(M[A])-|A|=\text{rank}(M + I_A) - \text{corank}(M[A])=r_{\langle \tau\delta \rangle}(M,A).$$
	
	\textbf{Case 5. $\bullet=\delta\tau\delta$ } 

For $\delta\tau\delta$, we have that
\[	\begin{aligned}
	(\delta\tau\delta)_A(M,I_n)\!=\!\delta_A \left(\!\!\!
	\begin{array}{c@{}c}
		\left[\!\!
		\begin{array}{cc|cc}
			M[A]\!+\!I_a \!\!\!&\!\! M[A,\! A^c] \!\!&\!\! M[A] \!\!&\!\! \mathbf{0}_{a,n-a}\\
			M[A^c,\! A] \!\!\!&\!\! M[A^c] \!\!&\!\! M[A^c \! ,\! A] \!\!&\!\! I_{n-a}\\
		\end{array}
		\!\!\right]
	\end{array}\!\right)\!\!=\!\!\!\begin{array}{c@{}c}
		\left[\!\!\!
		\begin{array}{cc|cc}
			M[A] \!\!\!\!&\!\! M[A, \! A^c] \!\!&\!\! M[A]\!+\! I_a \!\!&\!\! \mathbf{0}_{a,n-a}\\
			M[A^c \! ,\! A] \!\!\!\!&\!\! M[A^c] \!\!&\!\! M[A^c \! ,\! A] \!\!&\!\! I_{n-a}\\
		\end{array}
		\!\!\! \right]
	\end{array}
\end{aligned}
\]
by Definition \ref{defdeltatau}. Since $\begin{array}{c@{}c}
	\left[\!\!\!
	\begin{array}{cc}
		M[A] \!\!&\!\!\! M[A, \! A^c] \\
		M[A^c \! ,\! A] \!\!&\!\!\! M[A^c]\\
	\end{array}
	\!\!\! \right]=M
\end{array}$ and $\text{rank}\left(\!\!\!\begin{array}{c@{}c}
\left[\!\!\!
\begin{array}{cc}
	M[A]+I_a \!\!&\!\! \mathbf{0}_{a,n-a}\\
	M[A^c,A] \!\!&\!\! I_{n-a}\\
\end{array}
\!\!\! \right]
\end{array}\!\right)=\text{rank}(M[A]+I_a)+n-a=\text{rank}((M+I_n)[A])+n-a=n-\text{corank}((M+I_n)[A])$, one has that $$r_{\langle \delta\tau\delta \rangle}((M,I_n),A)=\text{rank}(M) - \text{corank}((M+I_n)[A])=r_{\langle \delta\tau\delta \rangle}(M,A)$$ by Equation (\ref{eq0}) and Definition \ref{def:matrix_partial_twuality}.
\end{proof}
\begin{Remark}
	Theorem \emph{\ref{Th:polyoftwuality}} proves the rank-statistic requirement in Problem \emph{\ref{Problem1}}.
\end{Remark}

By Case $5$ of the proof of Theorem \ref{Th:polyoftwuality}, the following Corollary \ref{Coro1} is an immediate consequence.
\begin{coro}\label{Coro1}
	Let $M$ be an $n\times n$ matrix over $\mathbb{GF}(2)$. For every \(A\subseteq V\),
	\[	\begin{aligned}
		(\delta\tau\delta)_A(M,I_n) =(M,I_n+MI_A)
	\end{aligned}
\quad\text{and}\quad I_A=\begin{array}{c@{}c}
		\left[\!\!\!
		\begin{array}{cc}
			I_{|A|} & \mathbf{0}_{|A|,n-|A|} \\
			\mathbf{0}_{n-|A|,|A|}
			 & \mathbf{0}_{n-|A|,n-|A|}\\
		\end{array}
		\!\!\! \right].
	\end{array}\]
\end{coro}

For $\bullet \in \{\delta, \tau, \delta\tau, \tau\delta, \delta\tau\delta\}$, the partial-\(\langle\bullet\rangle\) polynomial of a square matrix generalizes the corresponding polynomial of a bouquet. By Theorem~\ref{Th:polyoftwuality}, the paired-matrix polynomial is equivalent to the square-matrix polynomial and is therefore also a generalization of the bouquet polynomial.

\begin{coro}
	Let \(B\) be a bouquet, \(S=SI(B)\) be its signed intersection graph, and \((G,L)\) be the corresponding graft. Set \(M=\operatorname{adj}(G,L)\). Then, for every \(\bullet\in\{\delta,\tau,\delta\tau,\tau\delta,\delta\tau\delta\}\),
\[{}^{\partial}\varepsilon_B^\bullet(z)=P_{\langle\bullet\rangle}((M,I_{|E(B)|}),z).\]
	
\end{coro}
	

\section{Recurrence relations for the partial-$\langle\delta\tau\delta\rangle$~polynomial}

Recurrence relations for the partial-\(\langle\delta\rangle\) polynomial and the partial-\(\langle\tau\rangle\) polynomial of simple signed graphs, with respect to an edge, were obtained by Li~\cite{one-pointjoin} and by Feng, Guo, and Yan~\cite{Partial_Twuality}, respectively. In this section, we establish three recurrence relations for the partial-\(\langle\delta\tau\delta\rangle\) polynomial of a simple signed graph with respect to a specified edge. Each recurrence expresses the polynomial as a linear combination of polynomials of three or four simple signed graphs of smaller order.

Based on the different signs of two vertices associated with an edge in a signed graph, we have the following three theorems. To avoid redundancy, we introduce some notation that will be used throughout the proofs of the three recurrence relations. 

Let $M$ be an $n\times n$ symmetric matrix over $\mathbb{GF}(2)$, with rows and columns indexed by a finite set $V$. Fix distinct vertices \(u,v\in V\). Partition \(2^V\) into the four classes
\[
\begin{aligned}
\mathbb{V}_0&:=\{A\subseteq V:u\notin A,\ v\notin A\},&
\mathbb{V}_1&:=\{A\subseteq V:u\in A,\ v\notin A\},\\
\mathbb{V}_2&:=\{A\subseteq V:u\notin A,\ v\in A\},&
\mathbb{V}_3&:=\{A\subseteq V:u\in A,\ v\in A\}.
\end{aligned}
\]
By Corollary~\ref{Coro1},
\[
(\delta\tau\delta)_A(M,I_n)=(M,I_n+MI_A),
\]
and hence
\begin{equation}\label{eq1.9}
P_{\langle\delta\tau\delta\rangle}(M,z)
=z^{\operatorname{rank}(M)-n}
\sum_{A\subseteq V}z^{\operatorname{rank}(I_n+MI_A)}.
\end{equation}


\begin{thm}\label{recurrence1}
	Let $M$ be an $n\times n$ symmetric matrix over $\mathbb{GF}(2)$, with rows and columns indexed by a finite set $V$. Assume that $u,v\in V$ and $V^{\prime}=V\backslash\{u,v\}$. If 
	\[M=\left[
	\begin{array}{ccc}
		0 & 1 & P^T \\
		1 & 0 & Q^T \\
		P & Q & N \\
	\end{array}
	\right],\]
	where the rows and columns of the block matrix $M$ are indexed by $\{u\},\{v\}$ and $V^{\prime}$, then 
	\begin{equation*}
		\!\! P_{\langle \delta\tau\delta \rangle}(M,z) 
		\!=\!\! \left(\!z^{-\text{rank}(N)}\! P_{\langle \delta\tau\delta \rangle}(N,z)
		\! + \!
		z^{-\text{rank}(N+QQ^{T})}\! P_{\langle \delta\tau\delta \rangle}(N+QQ^{T}\!\!,z)
		\! + \! 
		z^{-\text{rank}({H})}\! P_{\langle \delta\tau\delta \rangle}({H},z) \!\right)\!\! z^{\!\text{rank}(M)}\!\!,
	\end{equation*}
where $H=\left[
\begin{array}{cc}
	1 & (P+Q)^{T} \\
	P+Q &  N+PP^{T}\\	
\end{array}
\right].$
\end{thm}

\begin{proof}
We consider $I_n+MI_A$ according to the different values of $i$ for which $A\in \mathbb{V}_i, i\in \{1,2,3,4\}$. 

\textbf{Case 1. $A\in \mathbb{V}_1$}

Since  $u,v\not\in A$, one has that
\[
I_n+MI_A=\left[\begin{array}{ccc}
	1 & 0 & P^{T}I_{A^{\prime}} \\
	0 & 1 & Q^{T}I_{A^{\prime}} \\
	\mathbf{0} & \mathbf{0} & I_{n-2}+NI_{A^{\prime}} \\
\end{array}
\right],\]
where $A^{\prime}=A$, $I_{A}=
\left[\begin{array}{ccc}
	0 & 0 &  \mathbf{0} \\
	0 & 0 &  \mathbf{0} \\
	\mathbf{0} & \mathbf{0} & I_{A^{\prime}} \\
\end{array}
\right]$ as $I_{A}$ (resp.~$I_{A^{\prime}}$) is indexed by $V$~($A\subseteq V$) (resp.~$V^{\prime}$~($A^{\prime}\subseteq V^{\prime}$)). It follows that $\text{rank}(I_n+MI_A)=2+\text{rank}(I_{n-2}+NI_{A^{\prime}})$ and by Equations (\ref{eq1.9}), we have that \[\sum_{A\in\mathbb{V}_1}z^{\text{rank}(I_n+MI_A)}=z^2\!\!\sum_{A^{\prime}\subseteq V^{\prime}}\!\!  z^{\text{rank}(I_{n-2}+NI_{A^{\prime}})}=z^{2-[\text{rank}(N)-(n-2)]}{P}_{\langle \delta\tau\delta \rangle}(N,z)=z^{n-\text{rank}(N)}{P}_{\langle \delta\tau\delta \rangle}(N,z).\]
Thus one has that
\begin{equation}\label{eq2}
 	z^{\text{rank}(M)-n} \sum_{A\in\mathbb{V}_1}  z^{\text{rank}(I_n+MI_A)}
 	=z^{\text{rank}(M)-\text{rank}(N)}{P}_{\langle \delta\tau\delta \rangle}(N,z).
 \end{equation}

\textbf{Case 2. $A\in \mathbb{V}_2\cup\mathbb{V}_3$}

If $A\in \mathbb{V}_2$, we have that $u\in A,v\not\in A$ and
\[I_n+MI_A=\left[\begin{array}{ccc}
	1 & 0 & P^{T}I_{A^{\prime}} \\
	1 & 1 & Q^{T}I_{A^{\prime}} \\
	P & \mathbf{0} & I_{n-2}+NI_{A^{\prime}} \\
\end{array}
\right],\]
where ${A^{\prime}}=A\backslash\{u\}$, $I_{A}=
\left[\begin{array}{ccc}
	1 & 0 &  \mathbf{0} \\
	0 & 0 &  \mathbf{0} \\
	\mathbf{0} & \mathbf{0} & I_{A^{\prime}} \\
\end{array}
\right]$ as $I_{A}$ (resp.~$I_{A^{\prime}}$) is indexed by $V$~($A\subseteq V$) (resp.~$V^{\prime}$~($A^{\prime}\subseteq V^{\prime}$)). We partition $I_n+MI_A$ as a new block matrix. Let 
$I_n+MI_A=\left[\begin{array}{cc}
	W\!\! &\!\! X\\
	Y \!\!&\!\! I_{n-2}+NI_{A^{\prime}} \\
\end{array}
\right],$ 
where $W=\left[\begin{array}{cc}
	1 & 0\\
	1 & 1 \\
\end{array}
\right],$
$X=\left[\begin{array}{c}
	P^{T}I_{A^{\prime}} \\
	Q^{T}I_{A^{\prime}} \\
\end{array}
\right],$ 
$Y=\begin{bmatrix} P  & \mathbf{0} \end{bmatrix}$. By Lemma \ref{calculaterank}, we can obtain that 
\[\begin{aligned}
	\text{rank}(I_n+MI_A)&=\text{rank}(W)+\text{rank}\left((I_{n-2}+NI_{A^{\prime}})-YW^{-1}X\right)\\
	&=2+\text{rank}\left((I_{n-2}+NI_{A^{\prime}})-YW^{-1}X\right)=2+\text{rank}\left((I_{n-2}+NI_{A^{\prime}})-PP^{T}I_{A^{\prime}}\right)\\
	&=2+\text{rank}\left(I_{n-2}+(N+PP^{T})I_{A^{\prime}}\right),
\end{aligned}\]
and by Equations (\ref{eq1.9}), one has
\begin{equation*}
	\sum_{A\in\mathbb{V}_2}z^{\text{rank}(I_n+MI_A)}=z^2 \sum_{A^{\prime}\subseteq V^{\prime}} z^{\text{rank}(I_{n-2}+(N+PP^{T})I_{A^{\prime}})}=z^{n-\text{rank}(N+PP^{T})}{P}_{\langle \delta\tau\delta \rangle}(N+PP^{T},z).
\end{equation*}
It follows that
\begin{equation}\label{eq3}
	z^{\text{rank}(M)-n}\sum_{A\in\mathbb{V}_2} z^{\text{rank}(I_n+MI_A)}
	=z^{\text{rank}(M)-\text{rank}(N+PP^{T})}{P}_{\langle \delta\tau\delta \rangle}(N+PP^{T},z).
\end{equation}
	
By symmetry of $u$ and $v$, we have 
\begin{equation}\label{eq4}
	z^{\text{rank}(M)-n} \sum_{A\in\mathbb{V}_3}  z^{\text{rank}(I_n+MI_A)}
	=z^{\text{rank}(M)-\text{rank}(N+QQ^{T})}{P}_{\langle \delta\tau\delta \rangle}(N+QQ^{T},z).
\end{equation}

\textbf{Case 3. $A\in \mathbb{V}_4$}

Since $u,v\in A$, we have that
\[I_n+MI_A=\left[\begin{array}{ccc}
	1 & 1 & P^{T}I_{A^{\prime}} \\
	1 & 1 & Q^{T}I_{A^{\prime}} \\
	P & Q & I_{n-2}+NI_{A^{\prime}} \\
\end{array}
\right],\]
where~${A^{\prime}}\!=\! A\backslash\{u,v\}$,~$I_{A}\!=\!
\left[\!\!\begin{array}{ccc}
	1 & 0 &\!  \mathbf{0} \\
	0 & 1 &\!  \mathbf{0} \\
	\mathbf{0} & \mathbf{0} &\! I_{A^{\prime}} \\
\end{array}\!\!\right]$~as~$I_{A}$~and~$I_{A^{\prime}}$~are~indexed~by~$V$~and~$V^{\prime}$,~respectively.~We~partition~$I_n+MI_A$~as a new block matrix.~Let $I_n+MI_A\! =\! \left[\!\!\begin{array}{cc}
	W \!\!&\!\! X\\
	Y \!\!&\!\! Z \\
\end{array}
\!\!\right],$ 
where 
$Z\! =\!\left[\!\!\begin{array}{cc}
	1 \!&\!\! Q^{T}I_{A^{\prime}} \\
	Q \!&\!\! I_{n-2}\!+\! NI_{A^{\prime}}
\end{array}
\!\!\right],$ 
$X=\left[\begin{array}{cc}
	1 & P^{T}I_{A^{\prime}} \\
\end{array}
\right],$ 
$Y=\left[\begin{array}{cc}
	1 & P \\
\end{array}
\right]^{T}$ and 
$W=\left[\begin{array}{c}
	1  \\
\end{array}
\right].$
By Lemma \ref{calculaterank}, we can obtain that 
\[\begin{aligned}
	\text{rank}(I_n+MI_A)&=\text{rank}(W)+\text{rank}(Z-YW^{-1}X)=\text{rank}(W)+\text{rank}(Z-YX)\\
	&=1+\text{rank}\left(\left[\begin{array}{cc}
		1 & Q^{T}I_{A^{\prime}} \\
		Q & I_{n-2}+NI_{A^{\prime}}
	\end{array}
	\right]-\left[\begin{array}{c}
		1 \\ P \\
	\end{array}
	\right]\left[\begin{array}{cc}
		1 & P^{T}I_{A^{\prime}}\\
	\end{array}
	\right]
	 \right)\\
	&=1+\text{rank}\left(\left[\begin{array}{cc}
		0 & (Q^{T}+P^{T})I_{A^{\prime}} \\
		P+Q & I_{n-2}+(N+PP^{T})I_{A^{\prime}}
	\end{array}
	\right]
	\right).
\end{aligned}
\]
Let a new index $r\notin V$ and $V^{\prime}$ be the index of the block
${H}=\left[
\begin{array}{cc}
	1 & (P+Q)^{T} \\
	P+Q &  N+PP^{T}\\	
\end{array}
\right]$.
One has \[\left[\!\!\begin{array}{cc}
	0 \!\! & \!\! (Q^{T}+P^{T})I_{A^{\prime}} \\
	P+Q \!\! & \!\! I_{n-2}+(N+PP^{T})I_{A^{\prime}}
\end{array}
\!\!\right]=
I_{n-1}+
\left[\!\!\begin{array}{cc}
	1 \!\! & \!\! Q^{T}+P^{T} \\
	P+Q \!\! & \!\! N+PP^{T}
\end{array}
\!\!\right]
\left[\!\!\begin{array}{cc}
	1 \!\! & \!\! \mathbf{0} \\
	\mathbf{0} \!\! & \!\!  I_{A^{\prime}}
\end{array}
\!\!\right]
=I_{n-1}+{H}I_{\{r\}\cup{A^{\prime}} }.
\]
It follows that 
\begin{equation*}
	\!\!\sum_{A\in\mathbb{V}_4}z^{\text{rank}(I_n+MI_A)}=z \sum_{A^{\prime}\subseteq V^{\prime}}\! z^{\text{rank}(I_{n-1}+{H}I_{\{r\}\cup A^{\prime}})}
	= z \sum_{B\subseteq \{r\}\cup V^{\prime}} z^{\text{rank}(I_{n-1}+{H}I_{B})}- z \sum_{A^{\prime\prime}\subseteq V^{\prime}} z^{\text{rank}(I_{n-1}+{H}I_{A^{\prime\prime}})}.
\end{equation*}
For any ${A^{\prime\prime}}\subseteq V^{\prime}$, we have that 
$I_{n-1}+{H}I_{A^{\prime\prime}}=I_{n-1}+{H}\left[\begin{array}{cc}
	0 \!\! & \!\! \mathbf{0} \\
	\mathbf{0} \!\! & \!\! I_{A^{\prime\prime}}
\end{array}
\right]=
\left[\begin{array}{cc}
	1 \!\! & \!\! (Q^{T}+P^{T})I_{A^{\prime\prime}} \\
	\mathbf{0} \!\! & \!\! I_{n-2}+(N+PP^{T})I_{A^{\prime\prime}}
\end{array}
\right],$ which implies that 
$\text{rank}(I_{n-1}+{H}I_{A^{\prime\prime}})=1+\text{rank}(I_{n-2}+(N+PP^{T})I_{A^{\prime\prime}})$
and \[\sum_{{A^{\prime\prime}}\subseteq V^{\prime}}z^{\text{rank}(I_{n-1}+{H}I_{{A^{\prime\prime}}})}=z\sum_{{A^{\prime\prime}}\subseteq V^{\prime}}z^{\text{rank}(I_{n-2}+(N+PP^{T})I_{A^{\prime\prime}})}.\]
Thus we have that
\begin{align*}
	\sum_{A\in\mathbb{V}_4}z^{\text{rank}(I_n+MI_A)}&= z \sum_{B\subseteq \{r\}\cup V^{\prime}} z^{\text{rank}(I_{n-1}+{H}I_{B})} - z^{2} \sum_{A^{\prime\prime}\subseteq V^{\prime}} z^{\text{rank}(I_{n-2}+{(N+PP^{T})}I_{A^{\prime\prime}})}\\
	&=z^{n-\text{rank}({H})}{P}_{\langle \delta\tau\delta \rangle}({H},z)-z^{n-\text{rank}((N+PP^{T}))}{P}_{\langle \delta\tau\delta \rangle}((N+PP^{T}),z)
\end{align*}
by Equation (\ref{eq1.9}), and  
\begin{equation}\label{eq5}
	\!\! z^{\text{rank}(M)-n}\!\sum_{A\in\mathbb{V}_4}\! z^{\text{rank}(I_n+MI_A)} \! = \! z^{\text{rank}(M)} \!\left(\! z^{-\text{rank}({H})}\! {P}_{\langle \delta\tau\delta \rangle}({H},z) \! - \! z^{-\text{rank}(N \! + \! PP^{T})}\! {P}_{\langle \delta\tau\delta \rangle}(N \! + \! PP^{T}\!\! ,z)\right). 
\end{equation}

Combined with Equations (\ref{eq2}-\ref{eq5}), we have that
\[\begin{aligned}
	\!\!\! P_{\langle \delta\tau\delta \rangle}(M,z)&\!=\! z^{\text{rank}(M)-n}\sum_{A\subseteq V(S)}z^{\text{rank}(I_n+MI_A)}=z^{\text{rank}(M)-n}\sum_{i=1}^{4}\sum_{A\in\mathbb{V}_i}z^{\text{rank}(I_n+MI_A)}	\\
	&\!=\! z^{\text{rank}(M)}\!\!\left(\! z^{-\text{rank}(N)}\! P_{\langle \delta\tau\delta \rangle}\!(N,z) \! + \! z^{-\text{rank}(N+QQ^{T})}\! P_{\langle \delta\tau\delta \rangle}\!(N+QQ^{T}\!\!,z) \! + \! z^{-\text{rank}({H})}\! P_{\langle \delta\tau\delta \rangle}\!({H},z)\!\right)\!.
\end{aligned}\]
\end{proof}

The three graphs with the matrix on the right-hand side of Theorem \ref{recurrence1} as their adjacency matrix can be obtained from the graph on the left by the following graph operations.

\begin{Remark}\label{Re1}
	Assume that $S$ is a simple signed graph with $\text{adj}(S)=M$. Let $V=V(S)$ and $u,v\in V$. The matrices $N,N+QQ^T,H$ can be seen as the adjacent matrix of three graphs obtained by applying some graph operations on $S$ as follows: $N=\text{adj}(S[V\backslash\{u,v\}])$, $N+QQ^T=\text{adj}\left(S \ast v\right)[V\backslash \{u,v\}]$ and $H=\text{adj}\left(S \ast u\right)[V\backslash \{u\}]$ by Remark \emph{\ref{keylemma}}.
\end{Remark}

\begin{thm}\label{recurrence2}\label{Re2}
	Let $M$ be an $n\times n$ symmetric matrix over $\mathbb{GF}(2)$, with rows and columns indexed by a finite set $V$. Assume that $u,v\in V$ and $V^{\prime}=V\backslash\{u,v\}$. If 
	\[M=\left[
	\begin{array}{ccc}
		1 & 1 & P^T \\
		1 & 0 & Q^T \\
		P & Q & N \\
	\end{array}
	\right],\]
	where the rows and columns of the block matrix $M$ are indexed by $\{u\},\{v\}$ and $V^{\prime}$, then 
	\[\begin{aligned}
		P_{\langle \delta\tau\delta \rangle}\!(M,z) 
		\!=\!\left(\!z^{-\text{rank}(M_v)}\! P_{\langle \delta\tau\delta \rangle}\! (M_v,z) \!+ \! z^{-\text{rank}(N+QQ^{T})}\!P_{\langle \delta\tau\delta \rangle}\!(N\! +\! QQ^{T}\!,z) \! + \! z^{-\text{rank}(L)}\! P_{\langle \delta\tau\delta \rangle}\! (L,z)\!\right)\!\! z^{\!\text{rank}(M)}\!,
	\end{aligned}	\]
	where $M_v=M[V\backslash \{v\}]$ and $L=N+PP^{T}+QP^T+PQ^T$.
\end{thm}

\begin{proof}
	We consider $I_n+MI_A$ according to the different values of $i$ for which $A\in \mathbb{V}_i, i\in\{1,2,3,4\}$.
	
	\textbf{Case 1. $A\in \mathbb{V}_1$}
	
	Since  $u,v\not\in A$, one has 
	$I_n+MI_A=\left[\begin{array}{ccc}
		1 & 0 & P^{T}I_{A^{\prime}} \\
		0 & 1 & Q^{T}I_{A^{\prime}} \\
		\mathbf{0} & \mathbf{0} & I_{n-2}+NI_{A^{\prime}} \\
	\end{array}
	\right],$ 
	where $A^{\prime}=A$ and $I_{A^{\prime}}$ is indexed by $V^{\prime}$ ($A^{\prime}\subseteq V^{\prime}$). 
	It follows that $\text{rank}(I_n+MI_A)=2+\text{rank}(I_{n-2}+NI_{A^{\prime}}).$
	Thus we have that
	\[\sum_{A\in\mathbb{V}_1}z^{\text{rank}(I_n+MI_A)}= z^2\sum_{A^{\prime}\subseteq V^{\prime}} z^{\text{rank}(I_{n-2}+NI_{A^{\prime}})} =  z^{2-[\text{rank}(N)-(n-2)]}{P}_{\langle \delta\tau\delta \rangle}(N,z) = z^{n-\text{rank}(N)}{P}_{\langle \delta\tau\delta \rangle}(N,z)\]
	and 
	\begin{equation}\label{eq8}
		z^{\text{rank}(M)-n}\sum_{A\in\mathbb{V}_1} z^{\text{rank}(I_n+MI_A)} =z^{\text{rank}(M)-\text{rank}(N)}{P}_{\langle \delta\tau\delta \rangle}(N,z).
	\end{equation}

	\textbf{Case 2. $A\in \mathbb{V}_2$}
	
	If $A\in \mathbb{V}_2$, we have that $u\in A,v\not\in A$ and
	\[I_n+MI_A=\left[\begin{array}{ccc}
		0 & 0 & P^{T}I_{A^{\prime}} \\
		1 & 1 & Q^{T}I_{A^{\prime}} \\
		P & \mathbf{0} & I_{n-2}+NI_{A^{\prime}} \\
	\end{array}
	\right]
	\rightarrow 
	\left[\begin{array}{ccc}
		1 & 1 & Q^{T}I_{A^{\prime}} \\
		0 & 0 & P^{T}I_{A^{\prime}} \\
		P & \mathbf{0} & I_{n-2}+NI_{A^{\prime}} \\
	\end{array}
	\right]
	\rightarrow 
	\left[\begin{array}{ccc}
		1 & 1 & Q^{T}I_{A^{\prime}} \\
		0 & 0 & P^{T}I_{A^{\prime}} \\
		\mathbf{0} & P & I_{n-2}+NI_{A^{\prime}} \\
	\end{array}
	\right]\]
	through elementary operations that swap rows and columns, where ${A^{\prime}}=A\backslash\{u\}$ and $I_{A^{\prime}}$ is indexed by $V^{\prime}$  ($A^{\prime}\subseteq V^{\prime}$). 
	One has $\text{rank}(I_n+MI_A)=1+\text{rank}(\Gamma(A))$, where
	$	\begin{aligned}
		\Gamma(A)=
		\begin{array}{c@{}c}
			\left[\begin{array}{cc}
				0 & P^{T}I_{A^{\prime}} \\
				P & I_{n-2}+NI_{A^{\prime}} \\
			\end{array}
			\right].
		\end{array}
	\end{aligned}$
	Assume that the block of $\Gamma(A)$ is indexed by a new index $r\notin V$ and $V^{\prime}$. 
	Since \[
	\Gamma(A)=I_{n-1}+
	\left[\begin{array}{cc}
		1 & P^{T} \\
		P & N \\
	\end{array}
	\right]
	\left[\begin{array}{cc}
		1 & \mathbf{0} \\
		\mathbf{0} & I_{A^{\prime}} \\
	\end{array}
	\right]
	=I_{n-1}+
	M[V(S)\backslash \{v\}]I_{{\{r\}\cup A^{\prime}}}=I_{n-1}+
	M_vI_{{\{r\}\cup A^{\prime}}},\]
	it follows that
	\begin{align*}
		\sum_{A\in\mathbb{V}_2}z^{\text{rank}(I_n+MI_A)}&=z\sum_{{A^{\prime}}\subseteq V^{\prime}}z^{\text{rank}(I_{n-1}+M_vI_{{\{r\}\cup A^{\prime}}})}\\
		&=z\sum_{{B}\subseteq\{r\}\cup V^{\prime}}z^{\text{rank}(I_{n-1}+M_vI_{B})} - z\sum_{{A^{\prime\prime}}\subseteq V^{\prime}}z^{\text{rank}(I_{n-1}+M_vI_{A^{\prime\prime}})}.
	\end{align*}
	
	For any ${A^{\prime\prime}}\subseteq V^{\prime}$, since \[I_{n-1}+M_vI_{{A^{\prime\prime}}}=I_{n-1}+
	\left[\begin{array}{cc}
		1 & P^{T} \\
		P & N \\
	\end{array}
	\right]
	\left[\begin{array}{cc}
		0 & \mathbf{0} \\
		\mathbf{0} & I_{A^{\prime\prime}} \\
	\end{array}
	\right]
	=\left[\begin{array}{cc}
		1 & P^{T}I_{A^{\prime\prime}} \\
		\mathbf{0} & I_{n-2}+NI_{A^{\prime\prime}} \\
	\end{array}
	\right],\]
	we can obtain that $\text{rank}(I_{n-1}+M_vI_{{A^{\prime\prime}}})=1+\text{rank}(I_{n-2}+NI_{A^{\prime\prime}})$ and
	\begin{equation*}
		\sum_{A^{\prime\prime}\subseteq V^{\prime}} z^{\text{rank}(I_{n-1}+M_vI_{{A^{\prime\prime}}})}=z\sum_{{A^{\prime}}\subseteq V^{\prime}}z^{\text{rank}(I_{n-2}+NI_{A^{\prime}})}.
	\end{equation*}
	Thus one has that
	\begin{align*}
		\sum_{A\in\mathbb{V}_2}z^{\text{rank}(I_n+MI_A)}&=z\sum_{{B}\subseteq \{r\}\cup V^{\prime}}z^{\text{rank}(I_{n-1}+M_vI_{B})} - z^2\sum_{{A^{\prime\prime}}\subseteq V^{\prime}}z^{\text{rank}(I_{n-2}+NI_{A^{\prime\prime}})}\\
		&= z^{n-\text{rank}({M_v})}{P}_{\langle \delta\tau\delta \rangle}({M_v},z) -  z^{n-\text{rank}(N)}\! {P}_{\langle \delta\tau\delta \rangle}(N,z)
	\end{align*}
	 by Equation (\ref{eq1.9}), and one has that
	 \begin{equation}\label{eq10}
	  z^{\text{rank}(M)-n}\sum_{A\in\mathbb{V}_2} z^{\text{rank}(I_n+MI_A)} = \left( z^{-\text{rank}({M_v})} {P}_{\langle \delta\tau\delta \rangle}({M_v},z)  -  z^{-\text{rank}(N)} {P}_{\langle \delta\tau\delta \rangle}(N,z)\right) z^{\text{rank}(M)}. 
	 \end{equation}

	\textbf{Case 3. $A\in \mathbb{V}_3$}
	
	Since $A\in \mathbb{V}_3$, we have that $u\notin A,v\in A$ and
	$I_n+MI_A=\left[\begin{array}{ccc}
		1 & 1 & P^{T}I_{A^{\prime}} \\
		0 & 1 & Q^{T}I_{A^{\prime}} \\
		\mathbf{0} & Q & I_{n-2}+NI_{A^{\prime}} \\
	\end{array}
	\right],$
	where ${A^{\prime}}=A\backslash\{v\}$ and $I_{{A^{\prime}}}$ is indexed by $V^{\prime}$  ($A^{\prime}\subseteq V^{\prime}$). 
	We partition $I_n+MI_A$ as a new block matrix. Let 
	$I_n+MI_A=\left[\begin{array}{cc}
		W & X\\
		Y & I_{n-2}+NI_{A^{\prime}} \\
	\end{array}
	\right],$
	where $W=\left[\begin{array}{cc}
		1 & 1\\
		0 & 1 \\
	\end{array}
	\right],$ 
	$X=\left[\begin{array}{c}
		P^{T}I_{A^{\prime}} \\
		Q^{T}I_{A^{\prime}} \\
	\end{array}
	\right],$ 
	$Y=\begin{bmatrix} \mathbf{0}  & Q \end{bmatrix}$. By Lemma \ref{calculaterank}, we can obtain that 
	\[\begin{aligned}
		\text{rank}(I_n+MI_A)&=\text{rank}(W)+\text{rank}\left((I_{n-2}+NI_{A^{\prime}})-YW^{-1}X\right)\\
		&=2+\text{rank}\left((I_{n-2}+NI_{A^{\prime}})-YW^{-1}X\right)=2+\text{rank}\left((I_{n-2}+NI_{A^{\prime}})-QQ^{T}I_{A^{\prime}}\right)\\
		&=2+\text{rank}\left((I_{n-2}+(N+QQ^{T})I_{A^{\prime}}\right)
	\end{aligned}
	\]
   by Equation (\ref{eq1.9}). Thus we have that
   \begin{equation*} 
   	\sum_{A\in\mathbb{V}_3}z^{\text{rank}(I_n+MI_A)}=z^2 \sum_{A^{\prime}\subseteq V^{\prime}} z^{\text{rank}(I_{n-2}+(N+QQ^{T})I_{A^{\prime}})}=z^{n-\text{rank}(N+QQ^{T})}{P}_{\langle \delta\tau\delta \rangle}(N+QQ^{T},z).
   \end{equation*}
   It follows that
   \begin{equation}\label{eq11}
   	z^{\text{rank}(M)-n}\sum_{A\in\mathbb{V}_3} z^{\text{rank}(I_n+MI_A)}
   	=z^{\text{rank}(M)-\text{rank}(N+QQ^{T})}{P}_{\langle \delta\tau\delta \rangle}(N+QQ^{T},z).
   \end{equation}

	\textbf{Case 4. $A\in \mathbb{V}_4$}
	
	Since $u,v\in A$, we have that
	$I_n+MI_A=\left[\begin{array}{ccc}
		0 & 1 & P^{T}I_{A^{\prime}} \\
		1 & 1 & Q^{T}I_{A^{\prime}} \\
		P & Q & I_{n-2}+NI_{A^{\prime}} \\
	\end{array}
	\right],$ 
	where ${A^{\prime}}=A\backslash\{u,v\}$ and $I_{A^{\prime}}$ is indexed by $V^{\prime}$  ($A^{\prime}\subseteq V^{\prime}$). 
	We partition $I_n+MI_A$ as a new block matrix. Let 
	$I_n+MI_A=\left[\begin{array}{cc}
		W & X\\
		Y & Z \\
	\end{array}
	\right],$
	where $W=\left[\begin{array}{cc}
		0 & 1  \\
		1 & 1\\
	\end{array}
	\right],$
	$X=\left[\begin{array}{c}
		P^{T}I_{A^{\prime}} \\
		Q^{T}I_{A^{\prime}} \\
	\end{array}
	\right],$ 
	$Y=\left[\begin{array}{cc}
		P & Q \\
	\end{array}
	\right]$ and 
	$Z=\left[\begin{array}{c}
		I_{n-2}+NI_{A^{\prime}}
	\end{array}
	\right].$ 
	Since $W^{-1}=\left[\begin{array}{cc}
		1 & 1  \\
		1 & 0\\
	\end{array}
	\right],$ by Lemma \ref{calculaterank},  we can obtain that 
	\[\begin{aligned}
		\text{rank}(I_{n}+MI_A)&=\text{rank}(W)+\text{rank}(Z-YW^{-1}X)\\
		&=2+\text{rank}\left(\left[\begin{array}{c}
			I_{n-2}+NI_{A^{\prime}}
		\end{array}
		\right]-\left[\begin{array}{cc}
			P+Q & P \\
		\end{array}
		\right]\left[\begin{array}{c}
			P^{T}I_{A^{\prime}} \\ Q^{T}I_{A^{\prime}} \\
		\end{array}
		\right]
		\right)\\
		&=2+\text{rank}\left(\left[\begin{array}{cc}
			I_{n-2}+(N+PP^{T}+QP^{T}+PQ^{T})I_{A^{\prime}}
		\end{array}
		\right]
		\right).
	\end{aligned}
	\]
	Recall that $L=N+PP^{T}+QP^{T}+PQ^{T}$, 
	one has that 
	$\text{rank}(I_n+MI_A)=2+\text{rank}(I_{n-2}+LI_{A^{\prime}}).$ 
	By Equation (\ref{eq1.9}), we have that
	\begin{equation*}
		\sum_{A\in\mathbb{V}_4}z^{\text{rank}(I_n+MI_A)}= z^2\sum_{A^{\prime}\subseteq V^{\prime}}  z^{\text{rank}(I_{n-2}+LI_{A^{\prime}})}= z^{n-\text{rank}(L)}{P}_{\langle \delta\tau\delta \rangle}(L,z)
	\end{equation*}
	and
	\begin{equation}\label{eq12}
		z^{\text{rank}(M)-n}\sum_{A\in\mathbb{V}_4} z^{\text{rank}(I_n+MI_A)} =z^{\text{rank}(M)-\text{rank}(L)}{P}_{\langle \delta\tau\delta \rangle}(L,z).
	\end{equation}
	
 Combined with Equations (\ref{eq8}-\ref{eq12}), the result holds.
\end{proof}
\begin{Remark}
	Assume that $S$ is a simple signed graph with $\text{adj}(S)=M$. Let $V=V(S)$ and $u,v\in V$. The matrices $M[V\backslash\{v\}],N+QQ^T,L$ can be seen as the adjacent matriices of three graphs obtained by applying some graph operations on $S$ as follows: $M[V\backslash\{v\}]=\text{adj}(S[V\backslash\{v\}])$, $N+QQ^T=\text{adj}\left(S \ast v\right)[V\backslash \{u,v\}]$ and $H=\text{adj}\left(S \ast v\ast u\right)[V\backslash \{u,v\}]$ by Remark \emph{\ref{keylemma}}.
\end{Remark}

\begin{thm}\label{recurrence3}
	Let $M$ be an $n\times n$ symmetric matrix over $\mathbb{GF}(2)$, with rows and columns indexed by a finite set $V$. Assume that $u,v\in V$ and $V^{\prime}=V\backslash\{u,v\}$. If 
	\[M=\left[
	\begin{array}{ccc}
		1 & 1 & P^T \\
		1 & 1 & Q^T \\
		P & Q & N \\
	\end{array}
	\right],\]
	where the rows and columns of the block matrix $M$ are indexed by $\{u\},\{v\}$ and $V^{\prime}$, then
	\[
	\begin{aligned}
		P_{\langle \delta\tau\delta \rangle}(M,z) 
		&=-z^{\text{rank}(M)-\text{rank}(N)}P_{\langle \delta\tau\delta \rangle}(N,z) +z^{\text{rank}(M)-\text{rank}(L)}P_{\langle \delta\tau\delta \rangle}(L,z)\\
		&+ 
		z^{\text{rank}(M)-\text{rank}(M_v)}P_{\langle \delta\tau\delta \rangle}(M_v,z)+ 	z^{\text{rank}(M)-\text{rank}(M_u)}P_{\langle \delta\tau\delta \rangle}(M_u,z),\\
	\end{aligned}	\]
	where $M_u=M[V\backslash \{u\}],M_v=M[V\backslash \{v\}]$ and  $L=N+PQ^{T}+QP^{T}$.
\end{thm}

\begin{proof}
	We consider $I_n+MI_A$ according to the different values of $i$ for which $A\in \mathbb{V}_i, i\in\{1,2,3,4\}$.

	\textbf{Case 1. $A\in \mathbb{V}_1$}
	
	Since $u,v\not\in A$, one has that $I_n+MI_A=\left[\begin{array}{ccc}
		1 & 0 & P^{T}I_{A^{\prime}} \\
		0 & 1 & Q^{T}I_{A^{\prime}} \\
		\mathbf{0} & \mathbf{0} & I_{n-2}+NI_{A^{\prime}} \\
	\end{array}
	\right],$
	where $A^{\prime}=A$ and $I_{A^{\prime}}$ is indexed by $V^{\prime}$  ($A^{\prime}\subseteq V^{\prime}$). It follows that $\text{rank}(I_n+MI_A)=2+\text{rank}(I_{n-2}+NI_{A^{\prime}})$ and	\[\sum_{A\in\mathbb{V}_1}z^{\text{rank}(I_n+MI_A)}\!=\! z^2\sum_{A^{\prime}\subseteq V^{\prime}} z^{\text{rank}(I_{n-2}+NI_{A^{\prime}})}\! = \! z^{2-[\text{rank}(N)-(n-2)]}{P}_{\langle \delta\tau\delta \rangle}(N,z) \!=\! z^{n-\text{rank}(N)}{P}_{\langle \delta\tau\delta \rangle}(N,z).\]
	Thus we have that
	\begin{align}\label{eq13.1}
		z^{\text{rank}(M)-n}\!\!\sum_{A\in\mathbb{V}_1}z^{\text{rank}(I_n+MI_A)} =z^{\text{rank}(M)-\text{rank}(N)}{P}_{\langle \delta\tau\delta \rangle}(N,z).
	\end{align}
	
	\textbf{Case 2. $A\in \mathbb{V}_2\cup \mathbb{V}_3$}
	
	If $A\in \mathbb{V}_2$, we have that $u\in A,v\not\in A$ and
	\[I_n+MI_A=\left[\begin{array}{ccc}
		0 & 0 & P^{T}I_{A^{\prime}} \\
		1 & 1 & Q^{T}I_{A^{\prime}} \\
		P & \mathbf{0} & I_{n-2}+NI_{A^{\prime}} \\
	\end{array}
	\right]
	\rightarrow 
	\left[\begin{array}{ccc}
		1 & 1 & Q^{T}I_{A^{\prime}} \\
		0 & 0 & P^{T}I_{A^{\prime}} \\
		P & \mathbf{0} & I_{n-2}+NI_{A^{\prime}} \\
	\end{array}
	\right]
	\rightarrow 
	\left[\begin{array}{ccc}
		1 & 1 & Q^{T}I_{A^{\prime}} \\
		0 & 0 & P^{T}I_{A^{\prime}} \\
		\mathbf{0} & P & I_{n-2}+NI_{A^{\prime}} \\
	\end{array}
	\right],\]
	through elementary operations that swap rows and columns, where ${A^{\prime}}=A\backslash\{u\}$ and $I_{A^{\prime}}$ is indexed by $V^{\prime}$  ($A^{\prime}\subseteq V^{\prime}$). 
	One has $\text{rank}(I_n+MI_A)=1+\text{rank}(\Gamma(A))$, where
	$	\begin{aligned}
		\Gamma(A)=
		\begin{array}{c@{}c}
			\left[\begin{array}{cc}
				0 & P^{T}I_{A^{\prime}} \\
				P & I_{n-2}+NI_{A^{\prime}} \\
			\end{array}
			\right].
		\end{array}
	\end{aligned}$ 
	Since $u\in A$ and 
	\[\Gamma(A)=I_{n-1}+
	\left[\begin{array}{cc}
		1 & P^{T} \\
		P & N \\
	\end{array}
	\right]
	\left[\begin{array}{cc}
		1 & \mathbf{0} \\
		\mathbf{0} & I_{A^{\prime}} \\
	\end{array}
	\right]
	=I_{n-1}+
	M[V\backslash \{v\}]I_{A}=I_{n-1}+
	M_vI_{A},\] 
    it follows that
	\begin{align*}
		\sum_{A\in\mathbb{V}_2}z^{\text{rank}(I_n+MI_A)}&=z\sum_{{B}\subseteq V\backslash \{v\},u\in B}z^{\text{rank}(I_{n-1}+M_vI_{B})}\\
		&=z\sum_{{B^{\prime}}\subseteq V\backslash \{v\}}z^{\text{rank}(I_{n-1}+M_vI_{B^{\prime}})}-z\sum_{{B^{\prime\prime}}\subseteq V^{\prime}}z^{\text{rank}(I_{n-1}+M_vI_{{B^{\prime\prime}}})}.
	\end{align*}
	For any ${B^{\prime\prime}}\subseteq V^{\prime}$, since 
	$I_{n-1}+M_vI_{B^{\prime\prime}} =I_{n-1}+M_v 
	\begin{array}{c@{}c}
		\left[\begin{array}{cc}
			0 & \mathbf{0} \\
			\mathbf{0} & I_{B^{\prime\prime}} \\
		\end{array}
		\right]
	\end{array}
	=
	\begin{array}{c@{}c}
		\left[\begin{array}{cc}
			1 & P^{T}I_{B^{\prime\prime}} \\
			\mathbf{0} & I_{n-2}+NI_{B^{\prime\prime}} \\
		\end{array}
		\right],
	\end{array}$
	we have
	\[\text{rank}(I_{n-1}+M_vI_{{B^{\prime\prime}}})=1+\text{rank}(I_{n-2}+NI_{B^{\prime\prime}}),\]
	which implies that 
	\begin{align*}
		\sum_{A\in\mathbb{V}_2}z^{\text{rank}(I_n+MI_A)}&=z\sum_{{B^{\prime}}\subseteq V\backslash \{v\}}z^{\text{rank}(I_{n-1}+M_vI_{B^{\prime}})}-z^2\sum_{{B^{\prime\prime}}\subseteq V^{\prime}}z^{\text{rank}(I_{n-2}+NI_{B^{\prime\prime}})}\\
		&=z^{n-\text{rank}(M_v)}{P}_{\langle \delta\tau\delta \rangle}(M_v,z)-z^{n-\text{rank}(N)}{P}_{\langle \delta\tau\delta \rangle}(N,z)
	\end{align*}
	by Equation (\ref{eq1.9}). Thus one has that
	\begin{align}\label{eq14.2}
		 z^{\text{rank}(M)-n}\sum_{A\in\mathbb{V}_2} z^{\text{rank}(I_n+MI_A)} = \left( z^{-\text{rank}({M_v})} {P}_{\langle \delta\tau\delta \rangle}({M_v},z)  -  z^{-\text{rank}(N)} {P}_{\langle \delta\tau\delta \rangle}(N,z)\right) z^{\text{rank}(M)}. 
	\end{align}
	By symmetry of $u$ and $v$, we have that
	\begin{equation}\label{eq14.3}
		 z^{\text{rank}(M)-n}\sum_{A\in\mathbb{V}_3} z^{\text{rank}(I_n+MI_A)} = \left( z^{-\text{rank}({M_u})} {P}_{\langle \delta\tau\delta \rangle}({M_u},z)  -  z^{-\text{rank}(N)} {P}_{\langle \delta\tau\delta \rangle}(N,z)\right) z^{\text{rank}(M)}. 
	\end{equation}

	\textbf{Case 3. $A\in \mathbb{V}_4$}
	
	Since $u,v\in A$, we have that
	$I_n+MI_A=\left[\begin{array}{ccc}
		0 & 1 & P^{T}I_{A^{\prime}} \\
		1 & 0 & Q^{T}I_{A^{\prime}} \\
		P & Q & I_{n-2}+NI_{A^{\prime}} \\
	\end{array}
	\right],$ 
	where ${A^{\prime}}=A\backslash\{u,v\}$ and $I_{A^{\prime}}$ is indexed by $V(S)\backslash \{u,v\}$  ($A^{\prime}\subseteq V^{\prime}$).
	We partition $I_n+MI_A$ as a new block matrix. 
	
	Let 
	$I_n+MI_A=\left[\begin{array}{cc}
		W & X\\
		Y & Z \\
	\end{array}
	\right],$
	where $W=\left[\begin{array}{cc}
		0 & 1  \\
		1 & 0\\
	\end{array}
	\right],$
	$X=\left[\begin{array}{c}
		P^{T}I_{A^{\prime}} \\
		Q^{T}I_{A^{\prime}} \\
	\end{array}
	\right],$ 
	$Y=\left[\begin{array}{cc}
		P & Q \\
	\end{array}
	\right]$ and 
	$Z=\left[\begin{array}{c}
		I_{n-2}+NI_{A^{\prime}}
	\end{array}
	\right].$ 
	By Lemma \ref{calculaterank}, since $W^{-1}=\left[\begin{array}{cc}
		0 & 1  \\
		1 & 0\\
	\end{array}
	\right],$ we can obtain that 
	\[\begin{aligned}
		\text{rank}(I+MI_A)&=\text{rank}(W)+\text{rank}(Z-YW^{-1}X)\\
		&=2+\text{rank}\left(\left[\begin{array}{c}
			I_{n-2}+NI_{A^{\prime}}
		\end{array}
		\right]-\left[\begin{array}{cc}
			Q & P \\
		\end{array}
		\right]\left[\begin{array}{c}
			P^{T}I_{A^{\prime}} \\ Q^{T}I_{A^{\prime}} \\
		\end{array}
		\right]
		\right)\\
		&=2+\text{rank}\left(\left[\begin{array}{cc}
			I_{n-2}+(N+QP^{T}+PQ^{T})I_{A^{\prime}}
		\end{array}
		\right]
		\right).
	\end{aligned}
	\]
	Recall that $N+QP^{T}+PQ^{T}=L$, by Equation (\ref{eq1.9}), one has that 
	\[\sum_{A\in\mathbb{V}_4}z^{\text{rank}(I_n+MI_A)} = z^2 \sum_{A^{\prime}\subseteq V^{\prime}}  z^{\text{rank}(I_{n-2}+LI_{A^{\prime}})} = z^{n-\text{rank}(L)}{P}_{\langle \delta\tau\delta \rangle}(L,z)\]
	and
	\begin{equation}\label{eq16}
		z^{\text{rank}(M)-n}\sum_{A\in\mathbb{V}_4} z^{\text{rank}(I_n+MI_A)} =z^{\text{rank}(M)-\text{rank}(L)}{P}_{\langle \delta\tau\delta \rangle}(L,z).
	\end{equation}
	
	Combined with Equations (\ref{eq13.1}-\ref{eq16}), the result holds. 
\end{proof}
\begin{Remark}\label{Re3}
	Assume that $S$ is a simple signed graph with $\text{adj}(S)=M$. Let $V=V(S)$ and $u,v\in V$. The matrices $N,M[V\backslash\{v\}],M[V\backslash\{u\}],L$ can be seen as the adjacent matrices of four graphs obtained by applying some graph operations on $S$ as follows: 
	$N=\text{adj}(S[V\backslash\{u,v\}])$, $M[V\backslash\{v\}]=\text{adj}(S[V\backslash\{v\}])$, $M[V\backslash\{u\}]=\text{adj}(S[V\backslash\{u\}])$  and $L=\text{adj}\left((((S+{\{u,v\}}) \ast u)\ast v)\ast u \right)[V\backslash \{u,v\}]$ by Remark \emph{\ref{keylemma},} where $S+\{u,v\}$ obtained by only changing the sign of vertices $u$ and $v$.
\end{Remark}
\begin{Remark}
	As the partial-\(\langle \delta\tau\delta  \rangle\) polynomial of matrices generalizes those of bouquets, the recurrence relations in Theorems \emph{\ref{recurrence1}-\ref{recurrence3}} can be used for the computation the polynomial of bouquets.
\end{Remark}

\section{Conclusion}

In this paper, we introduced a paired-matrix framework for the local operations \(\delta\) and \(\tau\). This framework provides a matrix realization of these operations without requiring the principal submatrices involved in the ordinary principal pivot transform to be nonsingular. More precisely, for every \(A\subseteq V\) and every \(\bullet\in\{\delta,\tau,\delta\tau,\tau\delta,\delta\tau\delta\}\), the rank-sum parameter
\[\operatorname{rank}(L)+\operatorname{rank}(R)-n,
\qquad (L,R)=\bullet_A(M,I_n),\]
recovers the exponent \(r_{\langle\bullet\rangle}(M,A)\) of the corresponding partial-twuality polynomial. Thus, our construction gives a rank-statistic answer to Problem \ref{Problem1}. We also established three recurrence relations for the partial-\(\langle\delta\tau\delta\rangle\) polynomial of simple signed graphs with respect to an edge (Theorems \ref{recurrence1}-\ref{recurrence3} and Remarks \ref{Re1}-\ref{Re3}). These recurrences reduce its computation to polynomials of smaller signed graphs and provide a practical tool for recursive evaluation.

\section*{Acknowledgment}

\indent
This work was partially supported by the National Natural Science Foundation of China
(Nos. 12471321 and 12331013), Fundamental Research Funds for the Central Universities (No.
KYAJBRC26004536) and Hebei Provincial Key Laboratory of Mathematical Theory and Analysis
for Network and Data Science.

\end{document}